%% file: StochZins1.tex
\documentclass[11pt,a4paper,final]{article}
\usepackage[latin1]{inputenc}
\usepackage{amsfonts}
\usepackage{amscd}
\usepackage{amsmath}
\usepackage{theorem}
\usepackage{mathrsfs}
\usepackage[dvips]{graphicx}
\usepackage{fancybox,amssymb}
\usepackage{geometry}
\usepackage[english]{babel}
\usepackage{indentfirst}

\newcommand{\mF}{\mathcal{F}}
\newcommand{\R}{\mathbb{R}}

\newcommand{\mP}{\mathbb{P}}
\newcommand{\mE}{\mathbb{E}}
\newcommand{\md}{\,{\rm d}}
\newcommand{\s}{\sum\limits}

\newcommand{\w}{\wedge}
\newcommand{\bq}{\begin{eqnarray*}}
\newcommand{\eq}{\end{eqnarray*}}
\newcommand{\one}{1\mkern-5mu{\hbox{\rm I}}}

\theoremstyle{break}
\theorembodyfont{}
\newtheorem{Def}{Definition}[section]
\newtheorem{Bem}[Def]{Remark}
\newtheorem{Lem}[Def]{Lemma}
\newtheorem{Satz}[Def]{Proposition}

\newtheorem{Bsp}[Def]{Example}
\newtheorem{out}[Def]{Outlook}

\makeatletter
\newenvironment{proof}{\noindent{\textit{Proof:}}}{%
\unskip\nobreak\hfil\penalty50\hskip1em\null\nobreak
$\Box$
\parfillskip=\z@\finalhyphendemerits=0\endgraf\bigskip}

\let\oldendBsp\endBsp
\def\endBsp{\unskip\nobreak\hfil\penalty50\hskip1em\null\nobreak\hfil%
$\blacksquare$\parfillskip=\z@\finalhyphendemerits=0\endgraf\oldendBsp}
\let\oldendBem\endBem
\def\endBem{\unskip\nobreak\hfil\penalty50\hskip1em\null\nobreak\hfil%
$\blacksquare$\parfillskip=\z@\finalhyphendemerits=0\endgraf\oldendBem}
\let\oldendout\endout
\def\endout{\unskip\nobreak\hfil\penalty50\hskip1em\null\nobreak\hfil%
$\blacksquare$\parfillskip=\z@\finalhyphendemerits=0\endgraf\oldendout}
\makeatother

\author{\small\sc  Julia Eisenberg \footnote{email: jeisenbe@fam.tuwien.ac.at}\smallskip\\\footnotesize Institute of Mathematical Methods in Economics, Vienna University of Technology.}
\date{}

\title{Unrestricted Consumption under a Deterministic Wealth and an Ornstein-Uhlenbeck Process as a Discount Rate}

\begin{document}
\maketitle
\begin{abstract}\noindent
We consider an individual or household endowed with an initial capital and an income, modeled as a linear function of time. Assuming that the discount rate evolves as an Ornstein-Uhlenbeck process, we target to find an unrestricted consumption strategy such that the value of the expected discounted consumption is maximized.
Differently than in the case with restricted consumption rates, we can determine the optimal strategy and the value function.

\vspace{6pt}
\settowidth\labelwidth{{\it 2010 Mathematical Subject Classification: }}%
                \par\noindent {\it 2010 Mathematical Subject Classification: }%
                \rlap{Primary}\phantom{Secondary}
                93B05\newline\null\hskip\labelwidth
                Secondary 49L20
\end{abstract}  
\input{intro}

\input{sec1}
\input{sec2copy}
\input{Bsp}
\subsection*{Acknowledgements}
\noindent
The research of the author was supported by the Austrian Science Fund, grant P26487.
\\The author would further like to thank Paul Kr\"uhner for helpful discussions.

\end{document}

%% file: intro.tex
\section{Introduction}
A pioneer of political economy Adam Smith said ``Consumption is the sole end and purpose of all production; [...]''. In fact, one of the fundamental questions in the decision theory is how an individual (or a household) should allocate her/his consumption over time and how much of an asset is it optimal to hold. The consumption behaviour (to save or to consume) depends on various factors, but for the main part on the individual's wealth and on the asset price processes. There is a variety of models investigating the problem of optimal consumption/investment under different assumptions about the wealth and asset price processes, confer for example \cite{kar}, \cite{busch} or \cite{fleming} and references therein.

Basically, the considered individual has a choice between consuming her/his wealth or investing in an asset in order to maximise, for example, the expected utility of consumption under a finite or infinite time horizon. Of course, the future cash flows should be transferred to the present through discounting. Usually, in order to simplify the calculations, the discount rate will be chosen as a deterministic constant, making the discount rate to the preference rate of the considered individual. 

But what happens if the individual's consumption will be discounted by a stochastic process? The problem of stochastic discounting under a linear deterministic wealth process has been considered in \cite{eis1} and under a Brownian motion as a surplus process in \cite{eis}. There, it was possible to find explicit expressions for the value function and the optimal strategy if the discounting function was given by a geometric Brownian motion. In this special case, it turned out that the stochastic discounting did not change the optimal strategy significantly compared to the case with a constant preference rate. In the case, the discount rate was given by an Ornstein-Uhlenbeck process and restricted consumption rates it was shown that the value function was a viscosity solution to the problem; but neither the value function nor the optimal strategy had been found.

In the present paper, we assume that the wealth process of the considered individual is given by a linear function of time and the short rate process is given by an Ornstein-Uhlenbeck process. We target to find the optimal unrestricted consumption strategy such that the expected discounted consumption is maximised. Of course, the assumption of a deterministic wealth process is not very realistic, but it allows to get a first idea of the influence of a stochastic interest rate on the consumption behaviour. A detailed discussion of the advantages and disadvantages of a stochastic interest rate by consumption maximisation problems would be very space-consuming and goes beyond the scope of this introduction. Also, in order to avoid unnecessary repetitions, we refer to \cite{eis} and to \cite{eis1}. 

The remainder of the paper is organized as follows.
In Section 2, we formulate and motivate the conjecture that the optimal strategy is of barrier type. In Section 3, we determine the optimal barrier strategy and prove via the verification theorem that the value function is a classical solution to the Hamilton--Jacobi--Bellman equation corresponding to the problem.
The results are illustrated by an example.

%% file: sec1.tex
\section{The Model}
Consider an individual or household with an income given by a deterministic linear function of time
\[
X_t=x+\mu t\;,
\]
$\mu>0$.
Denote further by $\{r_s\}$ an Ornstein-Uhlenbeck process
\[
r_s=re^{-a s}+\tilde b(1-e^{-a s})+\tilde\sigma e^{-as}\int_0^se^{au}\md W_u\;,
\]
where $\{W_u\}$ is a standard Brownian motion, $a,\tilde\sigma>0$, and let $U^r_s=\int_0^s r_u\md u$ with $r_0=r$.
Our target is to maximize the expected discounted consumption over all admissible strategies $C$, if the interest rate is given by $\{r_t\}$. A strategy $C$ is called admissible if $C$ is non-decreasing, adapted to the filtration $\{\mathcal F_s\}$, generated by $\{r_s\}$ and $X^C_t=X_t-C_t\ge 0$ for all $t\ge 0$.
\\
Here, we assume that the long-term mean $\tilde b$ of the process $\{r_s\}$ fulfils: $\tilde b> \frac{\tilde \sigma^2}{2a^2}$ and define 
\[
b:=\tilde b-\frac{\tilde \sigma^2}{2a^2} \quad\quad\mbox{and}\quad\quad \sigma:=\frac{\tilde \sigma}{\sqrt{2a}}\;.
\]
The return function corresponding to a strategy $C$ and the value function are defined by
\begin{align*}
&V^C(r,x)=\mE\Big[\int_0^\infty e^{-U^r_s} \md C_s|X_0=x,r_0=r\Big],\quad (r,x)\in\R\times\R_+\;,
\\& V(r,x)=\sup\limits_C V^C(r,x),\quad (r,x)\in\R\times\R_+\;.
\end{align*}
Note, that also lump sum payments are possible. The HJB equation corresponding to the problem is
\begin{equation}
\max\Big\{\mu V_x+a(\tilde b-r) V_r+\frac{\tilde \sigma^2}2 V_{rr}-r V,1-V_x\Big\}=0\;.\label{unrestricted}
\end{equation}
For the sake of convenience, we define an operator acting on sufficiently smooth functions
\begin{equation}
L(f)(r,x):=\mu f_x(r,x)+a(\tilde b-r) f_r(r,x)+\frac{\tilde \sigma^2}2 f_{rr}(r,x)-r f(r,x)\;.\label{oper}
\end{equation}
We conjecture that the optimal strategy would be of barrier type, i.e. it is optimal to consume if the short rate process exceeds some special value and to do nothing otherwise. 
Intuitively, it is clear that when starting with a negative initial discount rate, one should forego consumption, because the discounting factor $e^{-U_s^r}$ will increase at least until $\{r_t\}$ becomes positive. On the other hand, if $r_0>0$ then due to $\tilde b>0$ it could happen that $-U_s^r$ will remain negative and will keep decreasing in time. In this case it would make sense to start consuming on the maximal rate immediately. 
\\In order to find the optimal barrier we let $r^*\in\R$ be arbitrary but fixed and define
\[
\tau:=\inf\{t\ge 0:\; r_t=r^*,\; r_0=r<r^*\} \quad \mbox{and}\quad \varrho:=\inf\{t\ge 0:\; r_t=r^*,\; r_0=r>r^*\}\;
\]
and
\begin{align*}
&G(r,x):=\mE\Big[\big(x+\mu\tau+G(r^*,0))e^{-U_{\tau}^r}\Big]\quad \mbox{for $r\le r^*$},
\\&F(r,x):=x+\mE\Big[\mu\int_0^{\varrho}e^{-U_s^r}\md s+G(r^*,0)e^{-U_{\varrho}^r}\Big] \quad \mbox{for $r\ge r^*$}\;.
\end{align*}
Obviously, $G_x(r,x)=\mE\big[e^{-U_{\tau}^r}\big]$ and $F_x(r,x)=1$. 
Then, it is clear $G(r^*,x)=F(r^*,x)$ and $G_x(r^*,x)=F_x(r^*,x)$.
In order to verify whether a barrier strategy could be the optimal one, we have to investigate the properties of $G$ and $F$. Thus, we have to consider 
\begin{align*}
&\psi_1(r):=\mE[e^{-U_\tau^r}],\quad  r\in(-\infty,r^*],
\\&\psi_2(r):=\mE[\tau e^{-U_\tau^r}], \quad  r\in(-\infty,r^*],
\\&\phi_1(r):=\mE[e^{-U_\varrho^r}],\quad  r\in[r^*,\infty),
\\&\phi_2(r):=\mE\Big[\int_0^{\varrho}e^{-U_s^r}\md s\Big],\quad  r\in[r^*,\infty)\;.
\end{align*}
Similar to Shreve et al. \cite{shreve} we formulate the following lemma.
\begin{Lem}\label{lem:diff}
The functions $\psi_1$ and $\phi_1$ solve differential equation 
\begin{align}
a(\tilde b-r)f'(r)+\frac{\tilde \sigma^2}2f''(r)-rf(r)=0\;,\label{eqnr1}
\end{align}
the function $\phi_1$ solves
\begin{equation}
1+a(\tilde b-r)f'(r)+\frac{\tilde \sigma^2}2f''(r)-rf(r)=0\;.\label{eqnr2}
\end{equation}
and
$\psi_2$ solves 
\begin{equation}
\psi_1(r)+a(\tilde b-r)f'(r)+\frac{\tilde \sigma^2}2f''(r)-rf(r)=0\;.\label{eqnr5}
\end{equation}
under the boundary conditions 
\begin{align*}
&\psi_1(r^*)=1,\quad \lim\limits_{r\to-\infty}\psi_1(r)=0,
\\&\psi_2(r^*)=0,\quad \lim\limits_{r\to-\infty}\psi_2(r)=0,
\\&\phi_1(r^*)=1,\quad \lim\limits_{r\to\infty}\phi_1(r)=0,
\\&\phi_2(r^*)=0,\quad \lim\limits_{r\to\infty}\phi_2(r)=0.
\end{align*}
\end{Lem}
\begin{proof}
We prove the statement for $\psi_1$, the proof for $\phi_1$ follows with the same techniques.
\\$\bullet$ Let $f$ be a solution to Equation \eqref{eqnr1} with $f(r^*)=1$, $\lim\limits_{r\to-\infty} f(r)=0$. Then, by Ito's lemma
\begin{align*}
e^{-U_{\tau\w t}}f(r_{\tau\w t})&=f(r)+\int_0^{\tau\w t} e^{-U_s^r}\big\{a(\tilde b-r_s)f'(r_s)+\frac{\tilde\sigma^2}2f''(r_s)-r_sf(r_s)\big\}\md s
\\&\quad{}+\tilde\sigma\int_0^{\tau\w t} e^{-U_s^r}f'(r_s)\md W_s
\\&=f(r)+\tilde\sigma\int_0^{\tau\w t} e^{-U_s^r}f'(r_s)\md W_s\;.
\end{align*}
The integrand of the stochastic integral is bounded for $r_s\in(-\infty,r^*]$ so that the expectation of the stochastic integral equals zero, giving $f(r)=\mE\big[e^{-U^r_{\tau\w t}}f(r_{\tau\w t})\big]$.
Letting now $t\to\infty$ and noting that by Lebesgue's dominated convergence theorem limit and integration can be interchanged, we obtain
\[
f(r)=f(r^*)\mE\big[e^{-U^r_{\tau}}\big]=\mE\big[e^{-U^r_{\tau}}\big]\;.
\]
$\bullet$ Let now $f(r)$ solve Equation \eqref{eqnr2} with the boundary conditions $\lim\limits_{r\to\infty}f(r)=0$ and $f(r^*)=0$. Ito's formula yields
\begin{align*}
e^{-U^r_{\varrho\w t}}f(r_{\varrho\w t})=f(r)-\int_0^{\varrho\w t} e^{-U_s^r}\md s+\tilde\sigma\int_0^{\varrho\w t} e^{-U_s^r}f'(r_s)\md W_s, 
\end{align*}
giving $f(r)=\mE\big[e^{-U_{\varrho\w t}}f(r_{\varrho\w t})+\int_0^{\varrho\w t} e^{-U_s^r}\md s\big]$. Let now $t\to\infty$. Again by Lebesgue's dominated convergence theorem, we obtain
\[
f(r)=\mE\Big[\int_0^{\varrho} e^{-U_s^r}\md s\Big]\;.
\]
$\bullet$ Assume, $f$ solves Equation \eqref{eqnr5} with the boundary conditions $\lim\limits_{r\to-\infty}f(r)=0$ and $f(r^*)=0$. By Ito's formula
\begin{align*}
e^{-U^r_{\tau\w t}}f(r_{\varrho\w t})=f(r)-\int_0^{\tau\w t} e^{-U_s^r}\psi_1(r_s)\md s+\tilde\sigma\int_0^{\tau\w t} e^{-U_s^r}f'(r_s)\md W_s. 
\end{align*}
The expectation of the stochastic integral is equal to zero. For the first integral on the right side of the above equation, one gets due to the Markov property of $\{r_t\}$:
\begin{align*}
\mE\Big[\int_0^{\tau\w t} e^{-U_s^r}\psi_1(r_s)\md s\Big]&=\int_0^\infty\mE\Big[\one_{[s\le\tau\w t]} e^{-U_s^r}\psi_1(r_s)\Big]\md s
\\&=\int_0^\infty\mE\Big[\one_{[s\le\tau\w t]} \mE\big[e^{-U_\tau^r}|r_s\big]\Big]\md s
\\&=\int_0^\infty\mE\Big[ \mE\big[\one_{[s\le\tau\w t]} e^{-U_\tau^r}|\mF_s\big]\Big]\md s=\mE\big[(\tau\w t) e^{-U_\tau^r}\big]\;.
\end{align*}
Letting $t\to\infty$ and using Lebesgue's dominated convergence theorem yields the desired result.
\end{proof}
\begin{Bem}
Lemma \ref{lem:diff} implies that the functions $F$ and $G$ are twice continuously differentiable with respect to $r$, once continuously differentiable with respect to $x$ on $(r^*,\infty)\times \R_+$ and on $(-\infty,r^*)\times \R_+$ respectively and fulfil there
\begin{align*}
&\mu F_x(r,x)+ a(\tilde b-r)F_r(r,x)+\frac{\tilde \sigma^2}2F_{rr}(r,x)-rF(r,x)=-rx\;,
\\&\mu G_x(r,x)+ a(\tilde b-r)G_r(r,x)+\frac{\tilde \sigma^2}2G_{rr}(r,x)-rG(r,x)=0\;.
\end{align*}
In particular, $F$ solves the HJB equation on $[r^*,\infty)\times \R_+$ if $r^*\ge 0$. The function $G$ solves the HJB equation on $(-\infty,r^*]$ if $\psi_1(r)<1$ for $r<r^*$ and $\psi_1(r^*)=1$.  
\end{Bem}
\subsection{The function $G$ \label{G}}
Due to the properties of $\{r_t\}$, the hitting times $\tau$ and $\varrho$ are finite a.s. Note that $U^r_{t}=\frac {r-r_t}a+\tilde bt+\frac{\tilde \sigma} aW_{t}$. Thus, using the change of measure techniques, compare for instance \cite[p. 216]{HS}, with $\frac{\md \mP}{\md Q }=\exp\big(\frac{\tilde \sigma} a W_{\tau}+ \frac{\tilde \sigma^2}{2a^2}\tau\big)$ one obtains
\begin{align*}
\mE[e^{-U^r_{\tau}}]&=e^{-\frac {r-r^*}a}\mE[e^{-\tilde b\tau-\frac{\tilde \sigma} aW_{\tau}}]
=e^{-\frac {r-r^*}a}\mE_Q[e^{-\tilde b\tau-\frac{\tilde \sigma^2}{2a^2}\tau}]= e^{-\frac {r-r^*}a}\mE_Q[e^{- b\tau}]\;.
\end{align*}
Under the measure $Q$, the process $\{r_t\}$ has the long term mean $b=\tilde b-\frac{\tilde\sigma^2}{2a^2}$. In order to calculate $\mE[e^{-U^r_{\tau}}]$, we have to consider the Laplace transform of $\tau$. A parabolic cylinder function is defined as
\[
D_{-v}(y)= e^{-y^2/4} 2^{-v/2}\sqrt{\pi}\bigg\{\frac {1+\s_{k=1}^\infty \prod\limits_{j=0}^{k-1}\frac{(v+2j)y^{2k}}{(2k)!}}{\Gamma\big(\frac{v+1}2\big)}
-\frac{y\sqrt{2}\Big(1+\s_{k=1}^\infty \prod\limits_{j=0}^{k-1}\frac{(v+2j+1)y^{2k}}{(2k+1)!}\Big)}{\Gamma(v/2)}\bigg\},
\]
confer for example Borodin and Salminen \cite[p. 639]{bs}. By $\tilde D_v(y)$ we denote in the following the parabolic cylinder function $D_{-v}(y)$ multiplied by $e^{y^2/4}$.
In \cite[pp. 542]{bs}, one also finds the following formula
\begin{align*}
\mE_Q[e^{-b\tau}]=\frac{e^{\frac{(r-b)^2}{4\sigma^2}}D_{-b/a}\big(-\frac{r-b}{\sigma}\big)}{e^{\frac{(r^*- b)^2}{4\sigma^2}} D_{-b/a}\big(-\frac{r^*-b}{ \sigma}\big)}=\frac{\tilde D_{b/a}\big(-\frac{r-b}{\sigma}\big)}{\tilde D_{b/a}\big(-\frac{r^*-b}{ \sigma}\big)}\;,
\end{align*}
so that we can calculate $\mE[e^{-U^r_{\tau}}]$ explicitly. Note further that $e^{-b\tau}\tau$, $r<r^*$, is finite, so that one has 
\begin{align*}
\psi_2(r)=\mE[e^{-U^r_{\tau}}\tau]&=e^{-\frac {r-r^*}a}\mE_Q[e^{-b\tau}\tau]=e^{-\frac {r-r^*}a}\mE_Q\Big[\frac 1b \s_{n=1}^\infty\frac 1n\s_{k=0}^n\binom{n}{k}(-1)^k e^{-b\tau(k+1)}\Big]
\\&=\frac {e^{-\frac {r-r^*}a}}b \s_{n=1}^\infty\frac 1n\s_{k=0}^n\binom{n}{k}(-1)^k \frac{\tilde D_{b(k+1)/a}\big(-\frac{r-b}{\sigma}\big)}{\tilde D_{b(k+1)/a}\big(\frac{b-r^*}{\sigma}\big)}\;.
\end{align*}
Thus, for the function $G$ we find
\begin{align*}
G(r,x)
&=\big(x+G(r^*,0)\big)e^{-\frac {r-r^*}a}\frac{\tilde D_{b/a}\big(-\frac{r-b}{\sigma}\big)}{\tilde D_{b/a}\big(\frac{b-r^*}{\sigma}\big)}
\\&\quad{}+\frac {\mu e^{-\frac {r-r^*}a}} b \s_{n=1}^\infty\frac 1n\s_{k=0}^n\binom{n}{k}(-1)^k \frac{\tilde D_{b(k+1)/a}\big(-\frac{r-b}{\sigma}\big)}{\tilde D_{b(k+1)/a}\big(\frac{b-r^*}{\sigma}\big)}\;.
\end{align*}
In order to find an explicit expression for the function $F$, we have to find $\phi_1$ and $\phi_2$. For $\phi_1$ it holds due to \cite[pp. 542]{bs}
\[
\phi_1(r)=\mE[e^{-U^r_{\varrho}}]= e^{-\frac {r-r^*}a}\mE_Q[e^{-b\varrho}]=e^{-\frac {r-r^*}a}\frac{\tilde D_{b/a}\big(\frac{r-b}{\sigma}\big)}{\tilde D_{b/a}\big(\frac{r^*-b}{ \sigma}\big)}\;.
\]
To find $\phi_2$, we have to solve differential equation \eqref{eqnr2} and determine the coefficients with the help of the corresponding boundary conditions $\lim\limits_{r\to\infty}\phi_2(r)=0=\phi_2(r^*)$. 

%% file: sec2copy.tex
\section{The Optimal Strategy}
In order to obtain a continuously differentiable solution, we have to guarantee that the first derivatives of $F$ and $G$ coincide on $\{(r^*,x),\;x\in\R_+\}$. 
Obviously, it holds $G(r^*,x)=F(r^*,x)$, $G_x(r^*,x)=F_x(r^*,x)$. 
The derivative of $F$ with respect to $r$ does not depend on $x$. In order to have a continuously differentiable solution with respect to $r$, the derivative $G_r(r^*,x)$ should not depend on $x$. Thus, we have three conditions, yielding a continuously differentiable with respect to $x$ and to $r$ function, solving the HJB equation on $\R\backslash \{r^*\}\times \R_+$:
\begin{itemize}
\item $G_{rx}(r^*,x)=\psi_1'(r^*)=0$,
\item $\mE\Big[e^{-U_{\tau}^r}\Big]>1$ for all $r<r^*$,
\item $r^*\ge 0$.
\end{itemize} 
It holds
\begin{align*}
G_{rx}(r,x)=\psi_1'(r)&=\frac{\md }{\md r} e^{-\frac {r-r^*}a}\mE_Q\big[e^{-b\tau}\big]=\frac{\md }{\md r} e^{-\frac {r-r^*}a}\frac{\tilde D_{b/a}\big(\frac{ b-r}{\sigma}\big)}{\tilde D_{b/a}\big(\frac{b-r^*}{\sigma}\big)}
\\&=- \frac1a e^{-\frac {r-r^*}a}\frac{\tilde D_{b/a}\big(\frac{ b-r}{\sigma}\big)}{\tilde D_{b/a}\big(\frac{ b-r^*}{\sigma}\big)} +\frac b{a\sigma} e^{-\frac {r-r^*}a}\frac{\tilde D_{b/a+1}\big(\frac{ b-r}{\sigma}\big)}{\tilde D_{b/a}\big(\frac{ b-r^*}{\sigma}\big)}
\\&=- \frac1a e^{-\frac {r-r^*}a}\mE_Q\big[e^{-b\tau}\big]+\frac b{a\sigma}e^{-\frac {r-r^*}a}\frac{\tilde D_{b/a+1}\big(\frac{ b-r^*}{\sigma}\big)}{\tilde D_{b/a}\big(\frac{ b-r^*}{\sigma}\big)}\mE_Q\big[e^{-(b+a)\tau}\big]
\\&=e^{-\frac {r-r^*}a}\mE\Big[e^{-b\tau}\Big(-\frac 1a+\frac b{a\sigma}\frac{\tilde D_{b/a+1}\big(\frac{ b-r^*}{\sigma}\big)}{\tilde D_{b/a}\big(\frac{ b-r^*}{\sigma}\big)}e^{-a\tau}\Big)\Big]\;.
\end{align*}
Thus, if $\frac{\tilde D_{b/a+1}\big(\frac{ b-r^*}{\sigma}\big)}{\tilde D_{b/a}\big(\frac{b-r^*}{\sigma}\big)}=\frac {\sigma}b$ then $\psi_1'(r^*)= 0$ and $\psi_1(r)$ is strictly decreasing in $r$, which implies $\psi_1(r)>1$ for all $r<r^*$.  
\\In order to show the existence and uniqueness of $r^*$ with the properties described above, we have to consider the function
\[
H(y):=\frac{\tilde D_{b/a+1}\big(\frac{ b-y}{\sigma}\big)}{\tilde D_{b/a}\big(\frac{ b-y}{\sigma}\big)}\;.
\]
Since it is impossible, to determine the properties of the functions $\tilde D$ directly, we will derive the properties of $H$ from the differential equation corresponding to $\tilde D$.
\begin{Lem}
The function $H:\R\to\R_+$ is strictly increasing, convex, surjective and $H(0)<\frac{\sigma}b$.
\end{Lem}
\begin{proof}
Similar to Shreve et al. \cite{shreve} and Lemma \ref{lem:diff}, one can show that the function $h(r):=\mE[e^{-b\tau}]$ solves the following equation 
\[
\frac{\tilde\sigma^2}2 h''(r)+a(b-r)h'(r)=b h(r)\label{eqnr3} 
\]
with boundary conditions $h(r^*)=1$ and $\lim\limits_{r\to-\infty} h(r)=0$. Due to the properties of $\tau$, the function $h$ is strictly increasing. For the same reason, $h'(r)$ is strictly increasing: 
\[
h'(r)=\frac b{a\sigma}\frac{\tilde D_{b/a+1}\big(\frac{ b-r}{\sigma}\big)}{\tilde D_{b/a}\big(\frac{ b-r^*}{\sigma}\big)}=h'(r^*)\mE[e^{-(b+a)\tau}]\;.
\]
Since $b>0$, the function $h(r)$ does not have real zeros. Dividing \eqref{eqnr3} by $h(r)$ yields
\begin{equation}
\frac{\tilde\sigma^2}2 \frac{h''(r)}{h(r)}+a(b-r)\frac{h'(r)}{h(r)}=b\;.\label{eqnh}
\end{equation}
Note that $H(r)=\frac{\sigma a}b \frac{h'(r)}{h(r)}$.
\\Letting $r\to-\infty$ on the left side of the above equation yields $\lim\limits_{r\to-\infty} \frac{h'(r)}{h(r)}=0$, because otherwise the left hand side would become infinite. Thus, we can conclude
\[
\lim\limits_{r\to-\infty} H'(r)=\frac{\sigma a}b\lim\limits_{r\to-\infty} \Big\{\frac{h''(r)}{h(r)}-\frac{h'(r)^2}{h(r)^2}\Big\}\ge 0\;.
\]
On the other hand, we can rewrite the above equation in terms of $H$ and its derivatives
\[
\frac{\tilde \sigma ^2}2 H'(r)=b-a(b-r)H(r)-\frac{\tilde \sigma ^2}2 H(r)^2\;,
\]
which means $\frac{\tilde \sigma ^2}2 H''(r)=-a(b-r)H'(r)-\tilde \sigma ^2 H(r)H'(r)+aH(r)$. According to this one has $H''(r)>0$ if $H'(r)=0$, which implies  $H'(r)>0$ for $r\in\R$ due to $\lim\limits_{r\to-\infty}H'(r)\ge 0$.
\\$\bullet$ It holds $\lim\limits_{r\to\infty}\frac{h'(r)}{h(r)}=\infty$. Assume first $\lim\limits_{r\to\infty}\frac{h'(r)}{h(r)}=-A>-\infty$ for some $A\in\R_+$, i.e. $\lim\limits_{r\to\infty}\frac{h'(r)}{h(r)}=0$, which contradicts $H'(r)>0$. Assume now $\lim\limits_{r\to\infty}\frac{h'(r)}{h(r)}=B<\infty$, which gives $\lim\limits_{r\to\infty}H'(r)=0$. But Equation \eqref{eqnh} yields $\lim\limits_{r\to\infty}\frac{h''(r)}{h(r)}=\infty$, giving $\lim\limits_{r\to\infty}H'(r)=\lim\limits_{r\to\infty} \Big\{\frac{h''(r)}{h(r)}-\frac{h'(r)^2}{h(r)^2}\Big\}=\infty$.
\medskip
\\Thus, $\lim\limits_{r\to-\infty} H(r)=0$, $\lim\limits_{r\to-\infty} H(r)=\infty$. By the intermediate value theorem, we can conclude that $H(r)$ attains every value in $\R_+$.
\medskip
\\Inserting $r=0$ into Equation \eqref{eqnh} and multiplying \eqref{eqnh} by $\frac{\sigma}{b^2}$, yields
\begin{align*}
\frac\sigma b&=\frac{\tilde\sigma^2\sigma}{2b^2} \frac{h''(0)}{h(0)}+\frac{\sigma a}b\frac{h'(0)}{h(0)}
=\frac{\tilde\sigma^2\sigma}{2b^2} \frac{h''(0)}{h(0)}+H(0)\;.
\end{align*}
Since $\frac{\tilde\sigma^2\sigma}{2b^2} \frac{h''(0)}{h(0)}>0$, it holds $H(0)<\frac\sigma b$. 
\end{proof}
Due to the above lemma, there is a unique $r^*>0$ such that $H(r^*)=\frac{\sigma}b$, meaning that $G$ solves the HJB equation on $(-\infty,r^*]\times \R_+$ and $F$ solves the HJB equation on $[r^*,\infty)\times \R_+$. \medskip
\\\textbf{Assumption:} From now on, we assume that the functions $G$ and $F$ are defined for this special $r^*$.
\medskip
\\Letting $G(r^*,0)=\mu\frac{\psi_2'(r^*)- \phi_2'(r^*)}{\phi_1'(r^*)}$ guarantees $G_r(r^*,x)=F_r(r^*,x)$ for all $x\in\R_+$. 
It remains to show that
\begin{Lem}
The constant 
\[
\Delta:=\mu\frac{\psi_2'(r^*)- \phi_2'(r^*)}{\phi_1'(r^*)}
\]
is positive and finite.
\end{Lem}
\begin{proof}
Using the same change of measure technique like in Subsection \ref{G}, we obtain
\begin{align*}
\phi_1(r)=\mE\big[e^{-U_\varrho^r}\big]=e^{-\frac{r-r^*}a}\mE_Q\big[e^{-b\varrho}\big]\;.
\end{align*}
The stopping time $\varrho$ is increasing in $r$, implying that $\phi_1(r)$ is strictly decreasing in $r$, i.e. $\phi_1'(r^*)<0$.
\\As for the function $\phi_2$, it holds
\[
\phi_2(r)=\mE\Big[\int_0^\varrho e^{-U_s^r}\md s\Big]=e^{-\frac{r}a}\mE\Big[\int_0^\varrho e^{\frac{r}ae^{-as}-U_s^0}\md s\Big]\;,
\]
meaning that $e^{\frac{r}a}\phi_2(r)$ is strictly increasing in $r$. Thus, using that $\phi_2(r^*)=0$:
\[
\phi_2'(r^*)=-\frac 1a \phi_2(r^*)+e^{-\frac ra}\frac{\md}{\md r}\Big(e^{\frac{r}a}\phi_2(r)\Big)\Big|_{r=r^*}=e^{-\frac ra}\frac{\md}{\md r}\Big(e^{\frac{r}a}\phi_2(r)\Big)\Big|_{r=r^*}>0\;.
\]
We can rewrite the function $\psi_2$ like in the proof of Lemma \ref{lem:diff}
\begin{align*}
\psi_2(r)=\mE\Big[\int_0^\tau e^{-U^r_s}\psi_1(r_s)\md s\Big]\;.
\end{align*}
By the definition of $r^*$, the function $\psi_1(r)$ is strictly decreasing in $r$, which means that $\psi_1(r_s)$ is strictly decreasing in $r$. Since, $\tau$ and $-U_s^r$ are strictly decreasing in $r$, we can conclude that $\psi_2(r)$ is strictly decreasing, which proves the claim.
\end{proof}
\begin{Satz}
The optimal strategy $C^*$ is to immediately consume any capital bigger than zero if $r\ge r^*$, i.e. $C^*_t=\one_{[r\ge r^*]}X^*_{t-}$, where $\{X^*_t\}$ is the surplus process under the strategy $C^*$. The value function $V(r,x)$ is continuously differentiable with respect to $r$ and to $x$, twice continuously differentiable with respect to $r$ on $\R\backslash\{r^*\}\times \R_+$ and fulfils $V(r,x)=v(r,x)$ with
\[
v(r,x)=\begin{cases}
G(r,x) &: \mbox{$(r,x)\in (-\infty,r^*]\times \R_+$}\\
F(r,x) &: \mbox{$(r,x)\in (r^*,\infty)\times \R_+$}
\end{cases}.
\]
\end{Satz}
\begin{proof} 
Let $C$ be an arbitrary admissible strategy. Applying the fundamental theorem of calculus yields
\begin{align}
v(r_t,X_t^C)=v(r_t,x)+\int_0^t v_x(r_t,X^C_s)\md X^C_s\;.\label{b1}
\end{align}
In the following, we examine the two terms on the right side of the above equation.
Ito's formula requires $v$ to be twice continuously differentiable with respect to $r$, which is not fulfilled for $r=r^*$ and $x>0$.
Therefore, we use the extant second derivative Meyer-Ito formula \cite[p. 221]{protter}
where we just need $v$ to have an absolutely continuous derivative with respect to $r$ and $v_{rr}$ to be locally $L^1$. Since $F_r(r^*,x)=G_r(r^*,x)=1$ for all $x\in\R_+$, it is an easy exercise to verify that $v$ satisfies all above requirements. 
Then,
\begin{align}
v(r_t,x)=v(r,x)+\int_0^t v_r(r_s,x)\md r_s+\frac{\tilde\sigma^2}2\int_0^t v_{rr}(r_s,x)\md s\;. \label{b2}
\end{align}
Before we consider $v_x(r_t,X^C_s)$, note that $v_x$ does not depend on $x$ and it holds either $v_x=1$ or $v_x=\psi_1$. In particular, one can interchange the derivation order, i.e. $v_{rx}=v_{xr}$ and $v_{rrx}=v_{xrr}$. Like $v$, the function $v_x$ fulfils the conditions of the extant second derivative Meyer-Ito formula, \cite[p. 221]{protter}:
\begin{align*}
v_x(r_t,X_s^C)&=v_x(r_s,X_s^C)+\int_s^t v_{xr}(r_y,X_s^C)\md r_y+\frac{\tilde\sigma^2}2\int_s^t v_{xrr}(r_y,X_s^C)\md y
\\&=v_x(r_s,x)+\int_s^t v_{rx}(r_y,X_s^C)\md r_y+\frac{\tilde\sigma^2}2\int_s^t v_{rrx}(r_y,X_s^C)\md y\;. 
\end{align*}
Thus, integrating the above equality from $0$ to $t$ with respect to $\md X_s^C$ and applying Fubini's theorem yields
\begin{equation}
\label{b3}
\begin{split}
\int_0^t& v_x(r_t,X_s^C)\md X^C_s
\\&=\int_0^t \Big\{v_x(r_s,X_s^C)+\int_s^t v_{rx}(r_y,X_s^C)\md r_y+\frac{\tilde\sigma^2}2\int_s^t v_{rrx}(r_y,X_s^C)\md y \Big\}\md X^C_s
\\&=\int_0^t v_x(r_s,X_s^C)\md X^C_s+\int_0^t \big \{v_{r}(r_y,X^C_y)- v_{r}(r_y,x)\big\}\md r_y
\\&\quad {}+\frac{\tilde\sigma^2}2\int_0^t \big\{v_{rr}(r_y,X_y^C)- v_{rr}(r_y,x)\big\} \md y\;. 
\end{split}
\end{equation}
Thus, inserting \eqref{b2} and \eqref{b3} into \eqref{b1} yields
\begin{align*}
v(r_t,X_t^C)&= v(r,x)+\int_0^t v_r(r_s,X^C_s)\md r_s+\frac{\tilde\sigma^2}2\int_0^t v_{rr}(r_s,X_s^C)\md s+ \int_0^t v_x(r_s,X_s^C)\md X_s^C
\\&=v(r,x)+\int_0^t \mu v_x(r_s,X^C_s)+a(\tilde b-r_s)v_r(r_s,X^C_s)+\frac{\tilde\sigma^2}2 v_{rr}(r_s,X_s^C)\md s
\\&\quad {}+\tilde\sigma \int_0^t v_{r}(r_s,X_s^C)\md s - \int_0^t v_x(r_s,X_s^C)\md C_s\;.
\end{align*}
Via the product rule, using $v_x(r_s,X^*_s)\ge 1$ and $L(v)(r_s,X^C_s)\le 0$ we obtain
\begin{align*}
e^{-U_t^r}v(r_t,X_t^C)&=v(r,x)+ 
\int_0^t e^{-U_s^r}L(v)(r_s,X^C_s)\md s
+\tilde\sigma \int_0^t v_{r}(r_s,X_s^C)\md W_s
\\&\quad {}- \int_0^t e^{-U_s^r}v_x(r_s,X_s^C)\md C_s
\\&\le v(r,x)+\tilde\sigma \int_0^t e^{-U_s^r}v_{r}(r_s,X_s^C)\md W_s- \int_0^t e^{-U_s^r}\md C_s\;,
\end{align*}
with $L$ defined in \eqref{oper}.
Note that for the strategy $C^*$ equality holds.
Since the stochastic integral is a martingale, taking the expectations on the both sides of the above inequality yields
\begin{align}
\mE\Big[e^{-U_t^r}v(r_t,X_t^C)+\int_0^t e^{-U^r_s}\md C_s\Big]\le v(r,x)\;.\label{b5}
\end{align}
Consider now the first term in the expectation above. Since $v$ is increasing in $x$, one obtains
\begin{align*}
\mE\Big[e^{-U_t^r}v(r_t,X_t^C)\Big]\le \mE\Big[e^{-U_t^r}v(r_t,x+\mu t)\Big]\;.
\end{align*} 
From Subsection \ref{G}, we know that under the measure $Q$ it holds if $r<r^*$
\begin{align*}
&\psi_1(r)= e^{-\frac{r-r^*}a}\mE_Q[e^{-b\tau}]\le e^{-\frac{r-r^*}a},
\\&\psi_2(r)=\mE\Big[\int_0^\tau e^{-U_s^r}\psi_1(r_s)\md s\Big]\le \int_0^\infty e^{-bs} \mE_Q\big[e^{-\frac{r-r_s}a}e^{-\frac{r_s-r^*}a}\big]\md s \le\frac 1b e^{-\frac{r-r^*}a};
\end{align*}
analogously for $r\ge r^*$ one obtains $\phi_1(r)\le 1$ and $\phi_2(r)\le \frac 1b$.
Therefore, we can estimate
\begin{align*}
v(r_t,x+\mu t)&=\one_{[r_t\ge r^*]}F(r_t,x+\mu t)+\one_{[r_t< r^*]}G(r_t,x+\mu t)
\\&\le \big(1+e^{-\frac{r_t-r^*}a}\big)\big\{x+\mu t+\frac\mu b +\Delta\big\}\;.
\end{align*}
Using the measure $Q$ defined in Subsection \ref{G} and the fact that $r_t$ under $Q$ is normally distributed with mean $re^{-at}+b(1-e^{-at})$ and variance $\frac{\tilde\sigma^2}{2a}(1-e^{-2at})$, we obtain the following estimation
\begin{align*}
\mE\Big[e^{-U_t^r}v(r_t,X_t^C)\Big]&\le \mE\big[e^{-U_t^r}\big(1+e^{-\frac{r_t-r^*}a}\big)\big]\big\{x+\mu t+\frac\mu b +\Delta\big\}
\\&\le  e^{-\frac ra-bt}\mE_Q\big[\big(e^{\frac{r_t}a}+e^{\frac{r^*}a}\big)\big]\big\{x+\mu t+\frac\mu b +\Delta\big\}
\\&\le e^{-bt}e^{-\frac {r-b}a (1-e^{-at})+\frac {\tilde\sigma^2}{4a^3}(1-e^{-2at})}\big\{x+\mu t+\frac\mu b +\Delta\big\}\;.
\end{align*} 
Also, it holds 
\begin{align*}
\mE\Big[\int_0^t e^{-U^r_s}\md C_s\Big]\le \int_0^\infty \mE\big[e^{-U^r_s}\big]\md X_s<\infty\,,
\end{align*}
so that we can let $t\to\infty$ in \eqref{b5}, and obtain by Lebesgue's dominated convergence theorem
\[
v(r,x)\ge \mE\Big[\int_0^\infty e^{-U^r_s}\md C_s\Big]\;.
\]
\end{proof}

%% file: Bsp.tex
\begin{figure}[t]
\includegraphics[scale=0.35, bb = 0 300 100 400]{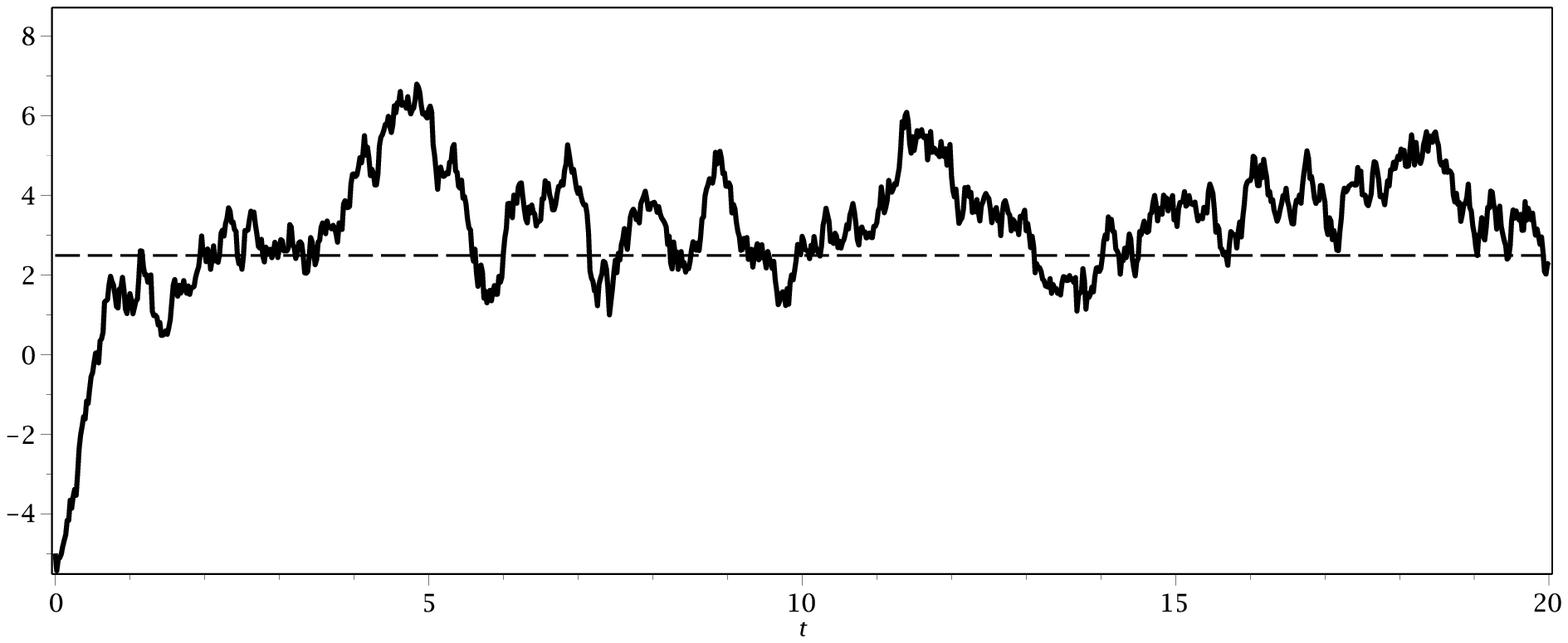}
\includegraphics[scale=0.35, bb = -450 300 -100 400]{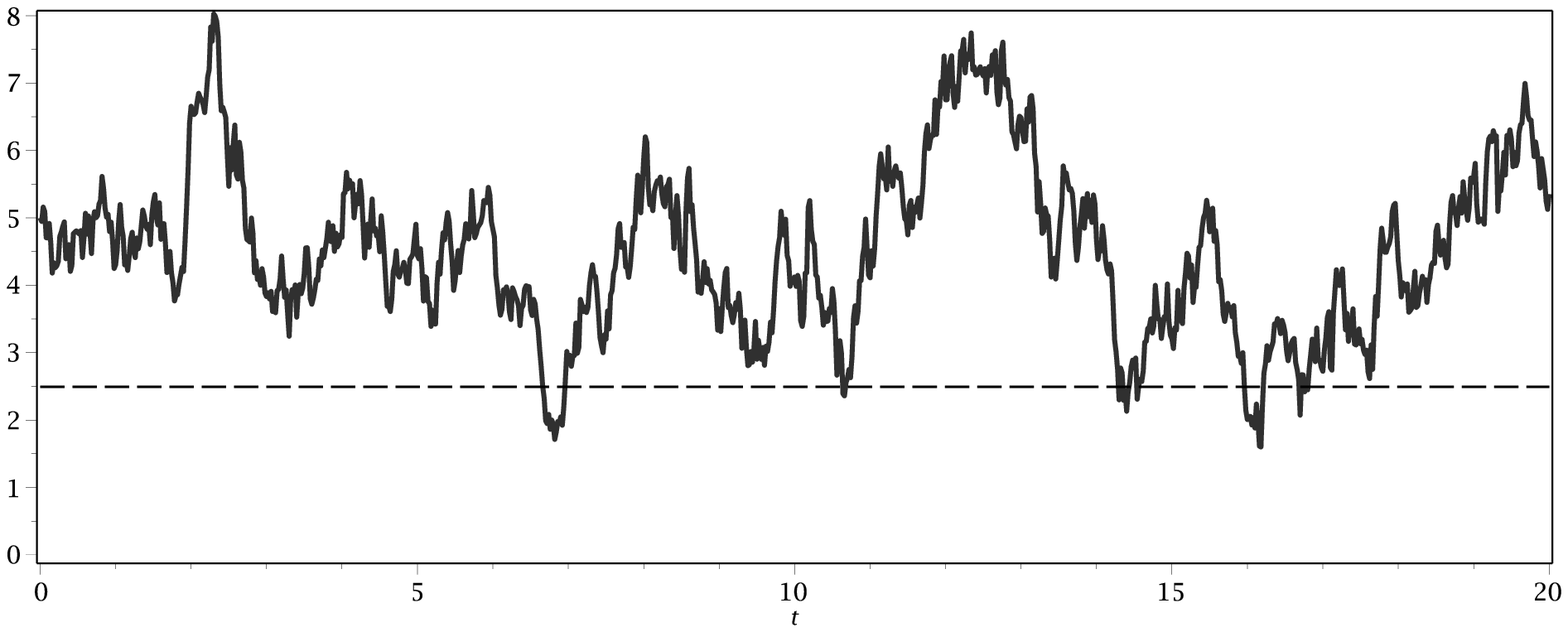}
\caption{Realisations of an OU-process with starting values $r_0=-5$ (the left picture) and $r_0=5$ and the optimal barrier $r^*=2.4936$ (dashed line).\label{real}}
\end{figure}
\begin{figure}[t]
\includegraphics[scale=0.6, bb = -150 50 200 250]{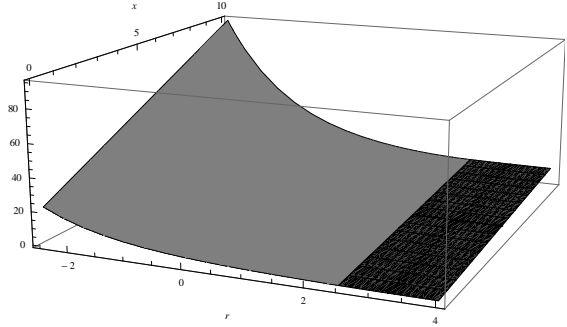}
\caption{The value function $V(r,x)$. The black and gray areas describe the strategies ``maximal consumption'' and ``no consumption'' correspondingly.\label{vf}}
\end{figure}
\begin{Bsp}
Let $a = 1$, $\tilde\sigma = 2$ and $\tilde b = 4$. 
The function $H(r)=\frac{\tilde D_{b/a+1}\big(\frac{b-r}{\sigma}\big)}{\tilde D_{b/a}\big(\frac{b-r}{\sigma}\big)}$ is strictly increasing and attains $\frac{\sigma}b=\frac 1{2\sqrt{2}}$ at $r^* = 2.4936$. 
In the time intervals where the process $\{r_t\}$ attains values smaller than $r^*$ it is optimal to wait, in the intervals where $r_t\ge r^*$ we pay everything. 
In Figure \ref{real} one can see realisations of an Ornstein-Uhlenbeck process (OU-process) with starting values $r_0=5$ and $r_0=-5$. Using the results from Subsection \ref{G} and solving differential equations \eqref{eqnr1} and \eqref{eqnr2} with corresponding boundary conditions, we can calculate the value function $V(r,x)$, illustrated in Figure \ref{vf}. 
\end{Bsp}